\documentclass[11pt,twoside]{amsart}
\usepackage{color}
\usepackage{hyperref}
\usepackage[latin1]{inputenc}
\usepackage[OT1]{fontenc}
\usepackage{amsmath}
\usepackage{amsthm}
\usepackage{amssymb}
\usepackage[all]{xy}
\usepackage{xy}
\usepackage{ifthen} 
\usepackage{hyperref} 
\usepackage{graphicx}
\usepackage{paralist}
\usepackage{stmaryrd}

\newtheorem{theorem}{Theorem}[section]

\newtheorem{proposition}[theorem]{Proposition}

\newtheorem{lemma}[theorem]{Lemma}

\numberwithin{equation}{section}

\theoremstyle{definition}
\newtheorem{definition}[theorem]{Definition}
\newtheorem{remark}[theorem]{Remark}

\newtheorem{assumption}[theorem]{Assumption}

\newcommand{\CC}{\mathbb{C}}

\newcommand{\NN}{\mathbb{N}}
\newcommand{\PP}{\mathbb{P}}

\newcommand{\RR}{\mathbb{R}}
\newcommand{\ZZ}{\mathbb{Z}}

\renewcommand{\to}{\xymatrix@1@=15pt{\ar[r]&}}
\renewcommand{\rightarrow}{\xymatrix@1@=15pt{\ar[r]&}}
\renewcommand{\mapsto}{\xymatrix@1@=15pt{\ar@{|->}[r]&}}
\renewcommand{\twoheadrightarrow}{\xymatrix@1@=15pt{\ar@{->>}[r]&}}
\renewcommand{\hookrightarrow}{\xymatrix@1@=15pt{\ar@{^(->}[r]&}}
\newcommand{\congpf}{\xymatrix@1@=15pt{\ar[r]^-\sim&}}

\begin{document}

\newboolean{xlabels} 
\newcommand{\xlabel}[1]{ 
                        \label{#1} 
                        \ifthenelse{\boolean{xlabels}} 
                                   {\marginpar[\hfill{\tiny #1}]{{\tiny #1}}} 
                                   {} 
                       } 
\setboolean{xlabels}{false} 

\title[Properties of dynamical degrees]{Some properties of dynamical degrees with a view towards cubic fourfolds}

\author[B\"ohning]{Christian B\"ohning$^1$}
\address{Christian B\"ohning, Mathematics Institute, University of Warwick\\
Coventry CV4 7AL, England}
\email{C.Boehning@warwick.ac.uk}

\author[Bothmer]{Hans-Christian Graf von Bothmer}
\address{Hans-Christian Graf von Bothmer, Fachbereich Mathematik der Universit\"at Hamburg\\
Bundesstra\ss e 55\\
20146 Hamburg, Germany}
\email{hans.christian.v.bothmer@uni-hamburg.de}

\author[Sosna]{Pawel Sosna$^2$}
\address{Pawel Sosna, Fachbereich Mathematik der Universit\"at Hamburg\\
Bundesstra\ss e 55\\
20146 Hamburg, Germany}
\email{pawel.sosna@math.uni-hamburg.de}

\thanks{$^1$ Supported by Heisenberg-Stipendium BO 3699/1-2 of the DFG (German Research Foundation) during the initial stages of this work}
\thanks{$^2$ Partially supported by the RTG 1670 of the  DFG (German Research Foundation)}

\dedicatory{This article is dedicated to Fedor Bogomolov on the occasion of his seventieth birthday.}

\begin{abstract}
Dynamical degrees and spectra can serve to distinguish birational automorphism groups of varieties in quantitative, as opposed to only qualitative, ways. We introduce and discuss some properties of those degrees and the Cremona degrees, which facilitate computing or deriving inequalities for them in concrete cases: (generalized) lower semi-continuity, sub-multiplicativity, and an analogue of Picard-Manin/Zariski-Riemann spaces for higher codimension cycles. We also specialize to cubic fourfolds and show that under certain genericity assumptions the first and second dynamical degrees of a composition of reflections in points on the cubic coincide.
\end{abstract}

\maketitle

\section{Introduction}\xlabel{sMotivation}

For a smooth projective variety $X$ of dimension $n$ and a birational self-map $f \colon X \dasharrow X$ one can define a tuple of real numbers
\[
\lambda_0 (f) = 1 ,\; \lambda_2 (f), \dots , \lambda_{n-1}(f), \; \lambda_n (f) =1,
\]
where $\lambda_i(f)\geq 1$ for $2\leq i\leq n-1$, called the dynamical degrees of $f$. See Section \ref{sPrelim} for two equivalent definitions of these numbers. The dynamical degrees turn out to be invariant under birational conjugacy, so that the dynamical spectrum
\[
\Lambda (X) = \{ \lambda (f) :=(\lambda_1 (f), \dots , \lambda_{n-1}(f)) \mid f \in \mathrm{Bir}(X) \} \subset \RR^{n-1}
\]
is a birational invariant of the variety $X$. One thus might, for example, try to use it to distinguish very general (conjecturally irrational) cubic fourfolds $X$ from $\PP^4$. Here are some ideas how the spectra might differ:
\begin{enumerate}
\item
As point sets, that is, there might be a tuple $\lambda (f)$ in the spectrum of $\PP^4$ which is not in the spectrum of $X$. This could for example be proven if one could show that on $X$ the dynamical degrees have to satisfy other additional inequalities, coming from the geometry of $X$, which can be violated on $\PP^4$. For this one has, in particular, to develop certain semi-continuity properties and computational tools for dynamical degrees, which we start doing in the subsequent sections.
\item
As metric (or only topological) spaces: for example, one can consider, for every $k$, the smallest gap in the spectrum after $1$, if there is any, that is,
\[
g_k:=\inf_{f \in \mathrm{Bir}(X), \; \lambda_k (f) \neq 1} (\lambda_k(f) -1).
\]
It is possible that these numbers differ for $X$ and $\PP^4$, but the drawback is that it is not easy to see how to relate them to accessible geometric features of $X$. In other words, it seems very hard to compute or estimate them.
On the topological side, it may happen that the Cantor-Bendixson ranks of the dynamical spectra or some linear slices in them differ, but again these are very hard to access.
\item
Arithmetically: it could happen that for both $X$ and $\PP^4$, the dynamical degrees are algebraic integers, but for $X$ they may satisfy some additional arithmetic constraints. For example, the number fields generated by each tuple of dynamical degrees on $X$ might differ from the ones for $\PP^4$. But also this seems very hard to detect.
\end{enumerate}

We thus concentrate on (1) for the moment.

The paper is organised as follows. In Section \ref{sPrelim} we recall some basic facts about cycles and dynamical degrees. In Section \ref{sSemiCont} we prove Theorem \ref{tSemiContinuity} which says that the dynamical degrees are lower-semicontinuous functions for families over \emph{smooth} bases, if one understands lower semi-continuity in a slightly generalized sense, namely that (downward) jumps may occur on countable unions of closed subsets. Moreover, in Proposition \ref{pCubicFourfolds} we collect some results that point out the importance of the questions of existence of (finite-dimensional, connected, non-trivial) algebraic subgroups in $\mathrm{Bir}(X)$ and of the classification of birational automorphisms with all dynamical degrees equal to $1$ for the irrationality problem for cubic fourfolds $X$. 
In Section \ref{sUpperDeg} we study the relationship between dynamical degrees and Cremona degrees under some assumptions and prove a sub-multiplicativity result, Theorem \ref{tSubMultiplicative}. Section \ref{sItTrans} is devoted to reflections on cubic fourfolds and in Theorem \ref{tGeneralPointsEquality} we show, under some assumptions, the equality of the first and second dynamical degrees of a composition of reflections on a smooth cubic fourfold. This is a sample of a type of result which says that special dynamical degrees can only arise in  the presence of special dynamics. Such implications are very important for making progress on the irrationality problem for cubics, too.
Finally, in the Appendix we describe generalized Picard-Manin spaces which might prove useful in the study of dynamical degrees on fourfolds in the future.  
\smallskip

\noindent\textbf{Conventions.} We work over the field of complex numbers $\mathbb{C}$ throughout the paper. A variety is a reduced and irreducible scheme of finite type over $\CC$. A prime $k$-cycle on a variety is an (irreducible) subvariety of dimension $k$. A $k$-cycle is a formal linear combination of prime $k$-cycles. If $f\colon X\dasharrow Y$ is a rational map between varieties, we denote by $\mathrm{dom}(f)$ the largest open subset of $X$ on which $f$ is a morphism. The graph $\Gamma_f \subset X \times Y$ of $f$ is the closure of the locus of points $(x, f(x))$ with $x \in \mathrm{dom}(f)$.
\smallskip

\noindent\textbf{Acknowledgments.} We would like to thank Miles Reid for useful discussions.

\section{Preliminaries}\xlabel{sPrelim}

\subsection{Cycles}
Let $X$ be a smooth projective variety of dimension $n$. Following \cite{Ful98}, we will denote by $A_k(X)$, $B_k(X)$ and $H_k(X)$ the groups of algebraic cycles of dimension $k$ modulo rational, algebraic and homological equivalence, respectively. Roughly speaking, a cycle is rationally equivalent to zero if it can be written as the difference of two fibres of a family over $\mathbb{P}^1$; algebraically equivalent to zero if it can be written as the difference of two fibres of a family over a smooth curve; and homologically equivalent to zero if it maps to zero under the cycle map $A_k(X)\to H^{2n-2k}(X,\mathbb{Z})$. 

Any zero cycle $\alpha\in A_0(X)$ has a well-defined degree, see \cite[Def.\ 1.4]{Ful98}. 

As usual, $A^k(X)$, $B^k(X)$, $H^k(X)$ will denote the groups of cycles of codimension $k$. There are surjections
\[A^k(X)\to B^k(X)\to H^k(X)\]
for all possible $k$.

Given a proper map $f\colon X\to Y$ of some relative dimension $l$, one can define pushforward maps $f_*\colon A_k(X)\to A_k(Y)$ and, if $f$ is also flat, pullback maps $f^*\colon A_k(Y)\to A_{k+l}(X)$, see \cite[Sect.\ 1]{Ful98}. For instance, for the pullback under a flat map one simply takes the class of the pre-image of a prime cycle and extends this definition linearly. One can also define pull-back maps for morphisms between smooth varieties which are not necessarily flat, but this is slightly more complicated and involves some intersection theory, see \cite[Ch.\ 8]{Ful98}.

Also recall that the collection of all groups $A_*(X)=\oplus_{k=0}^n A_k(X)$ is a commutative ring with respect to the intersection pairing (here we need $X$ smooth). More precisely, there are pairings $A_k(X)\times A_l(X)\to A_{k+l-n}(X)$ for all possible $k,l$. In other words, $A^*(X)$ is a graded ring.

An important result about cycles which we will use below is the ``Principle of Conservation of Number'', see \cite[Subsect.\ 10.2]{Ful98}.

\begin{theorem}\xlabel{tConsNumber}
Let $p\colon \mathcal{Y} \to T$ be a proper morphism of varieties, $\dim T=m$, and $\alpha$ be an $m$-cycle on $\mathcal{Y}$. Then the cycle classes $\alpha_t \in A_0 (\mathcal{Y}_t)$ for $t \in T$ a \emph{regular} closed point, all have the same degree. 
\end{theorem}

\subsection{Dynamical degrees}

Let $X$ be a smooth projective variety of dimension $n$, fix an ample divisor $H$ on $X$ and let $f\colon X \dasharrow X$ be a birational map. Setting $H^k_\mathbb{R}(X):=H^k(X)\otimes \mathbb{R}$, we define a linear map 
\[f^*\colon H^k_\mathbb{R}(X) \to H^k_\mathbb{R}(X),\quad \alpha\mapsto \mathrm{pr}_{1*}(\Gamma_f . \mathrm{pr}^*_2(\alpha)),\]
where $\mathrm{pr}_i\colon X\times X\to X$ are the projections and $\Gamma_f\subset X\times X$ is the graph of $X$.

\begin{definition}
Let $X$ and $f$ be as above. The \emph{$k$-th dynamical degree} of $f$ is 
\[\lambda_k (f) = \lim_{m\to \infty} ( \varrho ((f^m)^*))^{\frac{1}{m}},\]
where $\rho((f^m)^*)$ is the spectral radius of the linear map $(f^m)^*$.
\end{definition}

Equivalently, one can choose any resolution $Z$ of singularities of $\Gamma_f$, consider the induced diagram
\[
\xymatrix{
 & Z\ar[ld]_{\pi_1}\ar[rd]^{\pi_2} & \\
Y \ar@{-->}[rr]^f & & Y
}
\]
and define $f^* = \pi_{1*}\circ \pi_2^*$. For the equivalence of these two approaches, see \cite[Subsect.\ 3.1]{Truo15}.

There is yet another way of defining the dynamical degrees. First, one can define the \emph{Cremona degree} $\deg_k (f)$ of $f$ with respect to $H$ as the degree (that is, the intersection number with $H^{n-k}$) of the birational transform under $f^{-1}$ of a general element in the system of cycles homologically equivalent to $H^k$. In symbols, 
\[\deg_k(f)=H^{n-k} .  f^* H^k.\] 
One then puts
\[
\lambda_k (f) = \lim_{m \to \infty} \deg_k (f^m)^{\frac{1}{m}},
\]
the growth rate of the Cremona degrees of the iterates of $f$. This definition does not depend on the choice of $H$ and is equivalent to the previous one, see \cite[Thm.\ 1.1]{Truo15} for a statement valid over any algebraically closed field of characteristic zero or \cite[Thm.\ 2.4]{Guedj10} for a more analytic statement.

\begin{remark}\xlabel{rLimLiminf}
Before the existence of the limit was proven, a definition involving a limes inferior of Cremona degrees or spectral radii was sometimes used in the literature.
\end{remark}

The dynamical degrees satisfy several numerical properties, such as log-concavity, and are related to the topological entropy of $f$, see \cite{Guedj10}. 

\section{Semicontinuity of dynamical degrees}\xlabel{sSemiCont}

In this section we will prove Theorem \ref{tSemiContinuity}, providing estimates on dynamical degrees in families.

\begin{theorem}\xlabel{tSemiContinuity}
Let 
\[
\xymatrix{
X \ar[d]_{\pi} \ar@{^{(}->}[r] & \PP^N \times S\ar[ld]^{\mathrm{pr}_S} \\
S & 
}
\]
be a smooth projective family of smooth varieties $X_s \subset \PP^N$ over a \emph{smooth} variety $S$. Let 
\[
\xymatrix{
X \ar[rd]_{\pi } \ar@{-->}[rr]^{f} &    & X \ar[ld]^{\pi} \\
      & S & 
}
\]
be a birational map such that no entire fiber of $\pi$ is in the indeterminacy or exceptional locus of $f$, i.e., fiberwise, $f_s \colon X_s \dasharrow X_s$ is a well-defined birational map. If this is not satisfied at the beginning, we can always achieve it by replacing $S$ by a non-empty open subset. Consider the function
\begin{gather*}
\lambda_j \colon S \to \RR
\end{gather*}
associating to $s\in S$ the $j$-th dynamical degree $\lambda_j (f_s)$ of $f_s$. Then $\lambda_j$ is lower semi-continuous in the following generalized sense: the sets
\[
V_a:=\{ s \in S \mid \lambda_j (s) \le a \} , \quad a \in \RR ,
\]
are countable unions of Zariski closed subsets in $S$.
\end{theorem}

We begin with the following result, which generalises \cite[Lem.\ 4.1]{Xie15}.

\begin{lemma}\xlabel{lSemiContDeg}
In the setup of Theorem \ref{tSemiContinuity} let $H \subset \PP^N$ be the linear system of hyperplanes in $\PP^N$. Let $\deg_j (f_s)$ be the $j$-th Cremona degree of $f_s$. Then 
\[
\deg_j \colon S \to \mathbb{R} , \quad s \mapsto \deg_j (f_s)
\]
is a lower semi-continuous function on $S$. 
\end{lemma}

\begin{proof}
Consider the graph $\Gamma_f \subset X \times_S  X \subset \PP^N \times \PP^N \times S$ and the two (flat) projections $p_1, p_2$
\[
\xymatrix{
 & \PP^N \times \PP^N \times S  \ar[rd]^{p_2}\ar[ld]_{p_1} &  \\
\PP^N\times S &        & \PP^N \times S.
}
\]
Let $H^j_S \subset \PP^n\times S$ be the pull-back to $\PP^N \times S$ of the algebraic equivalence class $H^j$, that is, an intersection of $j$ relative hyperplanes in $\PP^N\times S$. Then, by Theorem \ref{tConsNumber} applied to $\mathcal{Y}=\PP^N \times \PP^N \times S$,$T = S$ and $\alpha = \Gamma_f . p_1^* (H_S^{\dim X_s -j}) . p_2^* (H_S^{j})$ (an easy computation shows that this is indeed a $\dim(S)$-cycle), the degree of $\alpha_t \in A_0 (\mathcal{Y}_t)$ is constant. Note that the cycle pull-backs via the flat morphisms (or, in any case, via morphisms between smooth varieties) and the intersection product in the ambient smooth variety $\PP^N \times \PP^N \times S$ are well-defined.
\smallskip

Consider $\Gamma_f \to S$ as an $S$-scheme. 
For $s \in S$ in a nonempty Zariski open subset $\Omega \subset S$, $(\Gamma_f)_s$ is equal to the graph of $f_s$ by an application of \cite[Lem.\ 2.3]{BBB15}, whereas for special $s \in Z:=S\backslash \Omega$, $\Gamma_{f_s}$ may be a proper component of $(\Gamma_f)_s$. On the other hand, if $(\Gamma_f)_s = \Gamma_{f_s}$, then $\alpha$  restricted to $(\Gamma_f)_s$ has degree $\deg_j (f_s)$.  

Thus we see that on $\Omega$, the function $\deg_j$ is constant. Let us check that $\deg_j$ can only get smaller at a point $z_0 \in Z$. This will imply the assertion. Indeed, if $\deg_j = a$ generically on $S$, there is a proper closed subset $S' \subset S$ such that $\{ s \in S \mid \deg_j (s) \le a-1 \} \subset S'$. Now considering the irreducible components of $S'$, and the restriction of the family $X$ to each of these, implies the assertion by induction on the dimension of $S$. 
\smallskip

To see that $\deg_j$ can only drop at $z_0 \in Z$, take a smooth curve $C \subset S$ with $z_0 \in C$ and $C \cap \Omega \neq \emptyset$, and restrict the family $\pi \colon X \to S$ to $C$, i.e. consider $\pi_C \colon X_C \to C$ and the restriction $f_C \colon X_C \dasharrow X_C$ of $f$ to $X_C$. Then we can do the above construction with $S$ replaced by $C$, that is, we can consider $\Gamma_{f_C}$ and a relative cycle $\alpha_C$. Note that since $\Omega \cap C \neq \emptyset$, for a general point $c\in C$ the degree of $(\alpha_C)_c$ will be equal to $\deg_j (f_c)$ (of course, $f_c=f_s$ for some general point $s\in S$). 

Now there is a finite set of points $\mathcal{P} \subset C$, where we can assume that $z_0 \in \mathcal{P}$, such that for $c\in C \backslash \mathcal{P}$, the fiber $(\Gamma_{f_C})_c$ of $\Gamma_{f_C}$ over $c$ is nothing but the graph of $(f_C)_c$. 

The graph $\Gamma_{f_C}$ is, by definition, the closure of $\Gamma_{f_C}\mid_{C\backslash \mathcal{P}} \to C \backslash \mathcal{P}$ in $\PP^N \times \PP^N \times C$. Now the advantage of working over a curve $C$ is that, by the valuative criterion of properness and the properness of Chow schemes of cycles in $\PP^N \times \PP^N$, the limits $(\Gamma_{f_C})_z$, for $z \in \mathcal{P}$, of the family $\Gamma_{f_C}\mid_{C\backslash \mathcal{P}} \to C \backslash \mathcal{P}$ will be cycles of dimension $\dim X_s$, that is, all components of $(\Gamma_{f_C})_{z_0}$ have dimension $\dim X_s$. This allows us to interpret the degree of the zero cycle $(\alpha_C)_{z_0}$ geometrically. Its degree is nothing but the sum of the numbers $\mathcal{G} . \mathrm{pr}_1^* (H^{\dim X_s -j}) . \mathrm{pr}_2^* (H^{j})$ running over the irreducible components $\mathcal{G}$ of $(\Gamma_{f_C})_{z_0}$, 
 possibly counted with suitable multiplicities. Here $\mathrm{pr}_i$ are the projections of $\PP^N\times \PP^N$ onto its factors. Since $\Gamma_{f_{z_0}}$ is one of the components, we see that the Cremona degree can only drop at a special point $z_0$.
\end{proof}

\begin{proof}(of Theorem \ref{tSemiContinuity})
By definition we have
\[
\lambda_j (f) = \lim_{n\to \infty } (\deg_j (f^n ))^{\frac{1}{n}}.
\]
Applying Lemma \ref{lSemiContDeg} to each iterate of $f$, we see that for each $n \in \NN$, there is a proper closed subset $Z_n \subset S$ such that on $S\backslash Z_n$, $\deg_j (f_s)$ is constant. Thus outside the countable union of proper closed subsets
\[
\bigcup_{k_0 \ge 1} \left( \bigcap_{k \ge k_0} Z_{k} \right) 
\]
the function $s \mapsto \lambda_j(f_s)$ is constant (here we use Remark \ref{rLimLiminf}). Now arguing again by induction on the dimensions of the irreducible components of the closed subsets in the previous countable union (these are countably many proper closed subsets) gives the assertion. 
\end{proof}

One instance where one can use the above considerations is the following result, which might be useful when trying to prove irrationality of very general cubic fourfolds. It is based on a suggestion by Miles Reid.

\medskip

Recall that given an injective map of (abstract) groups $\theta\colon G\, \hookrightarrow \mathrm{Bir}(X)$, we get an action of $G$ on $X$ by birational isomorphisms, and we say that $G$ is an \emph{algebraic subgroup} of $\mathrm{Bir}(X)$ if the domain of definition of the partially defined map $G\times X \dasharrow X$, $(g,x) \mapsto \theta(g)(x)$, contains a dense open set of $G \times X$ and coincides on it with a rational map (in the sense of algebraic geometry). By a theorem of Rosenlicht \cite[Thm.\ 2]{Ros56}, one can characterize algebraic subgroups $G$ equivalently by saying that they are those subgroups for which there is a birational model $Y \dasharrow X$ such that $G$ acts on $Y$ via biregular maps/automorphisms. 

\begin{proposition}\xlabel{pCubicFourfolds}
Let $X$ be a smooth cubic fourfold. For a line $l \subset X$ denote by $\varphi_l \colon \widetilde{X} \to \PP^2$ the associated conic fibration, and by 
\[
\mathrm{Bir}_{\varphi_l}(X) \subset \mathrm{Bir}(X)
\]
the subgroup consisting of birational self-maps which preserve the fibration $\varphi_l$, i.e. map a general fiber into itself. The following holds. 
\begin{enumerate}
\item
$X$ is rational if and only if there is an algebraic subgroup $(\CC^*)^2$ in $\mathrm{Bir}(X)$.
\item
If $X$ is rational, then there is a family of birational maps
\[
\xymatrix{
X \ar[rd]_{\pi } \ar@{-->}[rr]^{f} &    & X \ar[ld]^{\pi} \\
      & (\CC^*)^2 & 
}
\]
such that for any $s \in S:= (\CC^*)^2$ all dynamical degrees of $f_s:=f\mid_{X_s}$ are equal to $1$ and such that the set $\{ f_s \}_{s\in (\CC^*)^2}$ is contained in no subgroup $\mathrm{Bir}_{\varphi_l}(X)$ as above.
\item
If $X$ is rational, then there is some birational self-map $f \in \mathrm{Bir}(X)$ with all dynamical degrees equal to $1$ and such that $f$ is contained in no subgroup $\mathrm{Bir}_{\varphi_l}(X)$. There is even such a map $f$ with the property that the growth of the Cremona degrees of the iterates of $f$ is \emph{bounded}.
\end{enumerate}
\end{proposition}

\begin{proof}
For (1) note that if $X$ is rational, $\mathrm{Bir}(X)$ certainly contains a subgroup $(\CC^*)^2$, e.g., coordinate rescalings. Conversely, if there is an algebraic subgroup $(\CC^*)^2$ in $\mathrm{Bir}(X)$, then $X$ is birationally a $(\CC^*)^2$-principal bundle over a unirational, hence rational, surface (note that $(\CC^*)^2$, being abelian, will act generically freely once it acts faithfully). Since any such $(\CC^*)^2$-principal bundle is Zariski locally trivial, $X$ is rational in this case.

To prove (2), note that the existence of the family is clear by (1), so we only have to see that for any $s \in S:= (\CC^*)^2$ all dynamical degrees of $f_s=f\mid_{X_s}$ are equal to $1$. Now by Lemma \ref{lSemiContDeg}, clearly all Cremona degrees are bounded in the family. On the other hand, since it arises from a sub-\emph{group} $(\CC^*)^2$ of $\mathrm{Bir}(X)$, for any $f_s$, $f_s^n$ is again in the family, i.e. equal to some $f_t$. Hence the Cremona degrees of the iterates would be unbounded, unless all dynamical degrees are equal to $1$. Moreover, the set $\{ f_s \}_{s\in (\CC^*)^2}$ is contained in no subgroup $\mathrm{Bir}_{\varphi_l}(X)$, since the images of 
\[
(\CC^*)^2 \ni s \mapsto f_s(p) \subset X
\]
are rational \emph{surfaces} for a general point $p\in X$ if we construct the family $\{ f_s \}$ from a subgroup $(\CC^*)^2$ in $\mathrm{Bir}(X)$ as in (1). But if $\{ f_s \}_{s\in (\CC^*)^2}$ were  contained in a subgroup $\mathrm{Bir}_{\varphi_l}(X)$, these images would be conics.


To prove (3), suppose that
\[
\Phi \colon \PP^4 \dasharrow X
\]
is a birational map, and suppose that $x\in \PP^4$ is a point in which $\Phi$ is defined and a local isomorphism. Our goal is now to find a one-parameter group
\[
\{ f_t \}, \quad t \in \CC^* , \quad f_{t}\circ f_s = f_{ts}
\]
of birational self-maps of $X$ such that $\Gamma$, the (closure of the) image of 
\[
\CC^* \ni t \mapsto f_t (p) 
\]
is a curve on $X$ through $p$ which is \emph{not} a conic. Here we want all $f_t$ to be birational self-maps with all dynamical degrees equal to $1$. So $\Gamma$ is an orbit under the action of $\CC^*$ on $X$. If we can find such a one-parameter group resp. such a $\Gamma$ we are done: namely, if we choose $t$ to be of infinite order, then $f:=f_t$ has the required properties; indeed, if $f$ preserved a conic fibration, then the iterates $f^n (p)$, $n\in \NN$, would all have to lie in a conic $Q$. These iterates also lie on the curve $\Gamma$ and are distinct by our choice of $t$, thus $\Gamma$ and $Q$ would coincide (they have infinitely many points in common), a contradiction. 

One way to construct such a one-parameter subgroup is as follows: consider the family of all conics $\mathcal{C}_p$ through $p=\Phi (x)$. For every conic $C \in \mathcal{C}_p$ consider its birational transform $\Phi^{-1}_{\mathrm{bir}}(C)$ in $\PP^4$ through $x$. Now the degrees of the curves $\Phi^{-1}_{\mathrm{bir}}(C)$ for $C$ ranging over $\mathcal{C}_p$ are bounded above since degrees are lower-semicontinuous in families (compare also Lemma \ref{lCyclesDeg} below, and its proof). Assume without loss of generality that $x$ is the point $(1,1,1,1) \in \CC^4 \subset \PP^4$. 
Now consider the one parameter subgroups 
\[
\mathrm{diag}(t^a, t^b, t^c, t^d) \subset \mathrm{GL}_4 (\CC), \quad t \in \CC^*, \quad a,b,c,d \in \mathbb{N}. 
\]
These give families of birational maps of $\PP^4$ via rescaling the coordinates. The orbit closures of these subgroups (which are all isomorphic to $\CC^*$) are rational curves in $\PP^4$ whose degree is unbounded. Hence one of them is certainly not in the set of curves $\Phi^{-1}_{\mathrm{bir}}(C)$. This achieves our goal.
\end{proof}

\begin{remark}\xlabel{rBounded}
One might hope that birational self-maps $f$ with all dynamical degrees equal to $1$ are in some sense ``classifiable" on a very general cubic fourfold. Restricting this class even further, the maps whose iterates have bounded growth of the Cremona degrees can be studied. For surfaces, such maps $f$ are often called \emph{elliptic} and it is known, see \cite{B-C13}, that they are precisely those
that are \emph{virtually isotopic to the identity}, that is, on some model, an iterate $f^{n_0}$ belongs to the connected component of the identity of the (biregular) automorphism group. On varieties of higher dimension the term elliptic does
not seem so appropriate, so we call them \emph{bounded} maps here for now. Note that not only the dynamical degrees, but also the \emph{growth rate} of the sequence of Cremona degrees is invariant under birational conjugacy. So bounded maps
are a birationally invariant class. These maps have topological entropy equal to zero, hence are what is sometimes called topologically deterministic. If on a very general cubic fourfold it were true that any such map is a birational automorphism of finite order (we believe this could be true, at least none of the examples we know contradicts it), then Proposition \ref{pCubicFourfolds}.(3) would imply that a very general cubic fourfold is irrational. 
\end{remark}

\begin{remark}\xlabel{rConstant}
In part (3) of Proposition \ref{pCubicFourfolds}, one can even assume that the sequence of Cremona degrees is constant. This follows  from Lemma \ref{lSemiContDeg}: if $f_t$ is a family of birational self-maps of $X$ parametrized by an algebraic subgroup $\CC^* \subset \mathrm{Bir}(X)$, then the Cremona degrees are constant outside of a finite set of points in $\CC^*$. Now $f_s^n = f_{s^n}$, and for general $s$, the set of its powers $s^n$ avoids the previous finite set of points.
\end{remark}

\section{Upper bounds for dynamical degrees in terms of degrees}\xlabel{sUpperDeg}

For birational maps $f, g \colon \PP^n \dasharrow \PP^n$ the $k$-th Cremona degree $\deg_k$ is submultiplicative in the sense that 
\[
\deg_k (fg) \le \deg_k (f)\deg_k(g)
\]
whence also $\deg_k(f^n) \le \deg_k(f)^n$ and $\lambda_k(f) \le \deg_k (f)$; see \cite[Lem.\ 4.6]{RS97}. We want to investigate inhowfar the inequality $\lambda_k(f) \le \deg_k (f)$ remains valid in general. A version of quasi-submultiplicativity in the general case is proved in \cite[Prop.\ 2.6]{Guedj10}, but this does not directly yield good bounds because of the constant that shows up there. However, as we will see, another application of Theorem \ref{tConsNumber} easily yields a result for certain varieties more general than $\PP^n$.

We will use the following result.

\begin{lemma}\xlabel{lCycleTransformBirTransform}
\begin{enumerate}
\item
Let $f \colon X \to Y$ be a surjective morphism of smooth projective varieties. For any $l \in \mathbb{N}$, let
\[
Y_l := \{ y \in Y \mid \dim f^{-1}(y) \ge l \}, 
\]
and suppose that $V$ is a subvariety of $Y$ which intersects all irreducible components of the subvarieties $Y_l$ as well as their mutual intersections properly. Then the pull-back class $f^*[V]$ in $A_* (X)$ is represented by $f^{-1}(V)$ which is equal to the birational transform of $V$ under $f^{-1}$.
\item
Now let $g \colon X_1 \dasharrow X_2$ be a birational map. Pick a resolution of $g^{-1}$
\[
\xymatrix{
 & \widetilde{X}_2\ar[rd]^q\ar[ld]_p & \\
 X_1 \ar@{-->}[rr]^g &  & X_2
}
\]
where $q$ is a succession of blowups in subvarieties lying over components of the base locus $\mathrm{Bs}(g^{-1})\subset X_2$. Suppose that $Z$ is a subvariety which satisfies the conditions for $V$ in part (1) with respect to the morphism $q$ and which is also not contained in $\mathrm{Exc}(g^{-1})$, the union of the subvarieties in $X_2$ contracted by $g^{-1}$. Then $g^*[Z]$ is represented by the birational transform of $Z$ under $g^{-1}$. 
\end{enumerate}
\end{lemma}

\begin{proof}
Part (1) is \cite[Thm.\ A.5 and Thm.\ 1.23]{EiHa15} or  \cite[Lem.\ 3.1]{Truo15} and proof of Lem.\ 3.2(a) ibid. For part (2) we can apply part (1) to conclude that $q^*[Z]$ is represented by the birational transform of $Z$ under $q^{-1}$ on $\widetilde{X}_2$. But $g^*[Z]$ is equal to $p_*q^*[Z]$, and $p_*$ is an isomorphism onto its image on $q^*[Z]$ outside of the strict transforms on $\widetilde{X}_2$ of the subvarieties contracted by $g$, hence the claim. 
\end{proof}

\begin{definition}
We will call a $Z$ as in part (2) of Lemma \ref{lCycleTransformBirTransform} $g$-\emph{adapted}.
\end{definition} 

\begin{theorem}\xlabel{tSubMultiplicative}
Let $X$ be a smooth projective variety such that $H^k_{\RR}(X)$ is one-dimensional generated by $H^k$ where $H\subset X$ is a very ample divisor. For any birational map $f\colon X \dasharrow X$ consider the map $f^* \colon H^k_{\RR}(X) \to H^k_{\RR}(X)$ defined in Section \ref{sPrelim}. Suppose this map is multiplication by $\delta_k (f) \in \NN$. Also suppose that for any birational map $g \colon X \dasharrow X$, there is an $m\in \NN$ and a flat family $p\colon \mathcal{C} \to S$ of effective irreducible codimension $k$ cycles homologically equivalent to $\delta_k(f)mH^k$ over a smooth irreducible base $S$ such that $\mathcal{C}_{s_0}$ is the birational transform under $f^{-1}$ of a general element in $mH^k$, and for general $s\in S$, $\mathcal{C}_s$ is $g$-adapted in the sense of Lemma \ref{lCycleTransformBirTransform} (2). 
Then 
\[
\delta_k (f \circ g) \le \delta_k (f) \cdot \delta_k (g).
\]
In particular, $\lambda_k(f) \le \delta_k (f)$ for such a birational map $f \colon X \dasharrow X$. 
\end{theorem}

To prove Theorem \ref{tSubMultiplicative} we use the following Lemma.

\begin{lemma}\xlabel{lCyclesDeg}
Let $X$ be a smooth and projective variety of dimension $n$ with a very ample class $H$, $S$ be a smooth variety and 
let $\mathcal{C} \to S$ be a flat family of irreducible $k$-cycles of some fixed degree on $X \times S$, $\iota\colon  \mathcal{C} \hookrightarrow X \times S$ be the inclusion, and $g\colon X \dasharrow X$ a birational map such that $g^{-1} \circ \iota_s$ is a well-defined birational map on every fiber $\mathcal{C}_s$. Then the function
\[
s \mapsto \deg ( (g^{-1}\circ \iota_s)_* (\mathcal{C}_s))
\] 
which associates to a point $s\in S$ the degree of the birational transform of the cycle $\mathcal{C}_s$ under $g^{-1} \circ \iota_s$ (with respect to $H$) is a lower semi-continuous function on $S$.
\end{lemma}

\begin{proof}
Again we employ Theorem \ref{tConsNumber}. Let $\mathcal{Y} = X \times X \times S$ with projections $p_1:=p_{13}\colon X\times X \times S \to X \times S$ and $p_2=p_{23} \colon X\times X \times S \to X\times S$. Moreover, for the fiberwise birational map $g^{-1} \times \mathrm{id}_S \colon X \times S \dashrightarrow X \times S$, consider the composite 
\[(g^{-1}\times\mathrm{id}_S)\circ \iota \colon \mathcal{C} \hookrightarrow X \times S ,\]
and (the closure in $\mathcal{Y}$ of) its graph $\Gamma_{(g^{-1}\times\mathrm{id}_S)\circ \iota}$.

Consider the cycle $\alpha=\Gamma_{(g^{-1}\times\mathrm{id}_S)\circ \iota} . p_2^* (H^{i}_S)$ on $\mathcal{Y}$, where $i$ is chosen appropriately such that $\alpha_s$ is a zero cycle for all $s$. Then arguing as in the proof of Lemma \ref{lSemiContDeg}, we find that for general $s$ in $S$, the degree of  $\alpha_s$ is equal to $ \deg ( (g^{-1}\circ i_s)_*(\mathcal{C}_s))$, whereas as special points it only gives an upper bound. Since the degree of $\alpha_s$ is constant by Theorem \ref{tConsNumber}, we get the assertion.
\end{proof}

\begin{proof}(of Theorem \ref{tSubMultiplicative})
Replacing $H^k$ by $mH^k$ we can assume $m=1$. Then $\delta_k (f)$ is the degree of the birational transform under $f^{-1}$ of a general element $H^k$, divided by $H^n$, and the same holds for $g$. In symbols: $\frac{f^{*}H^k.H^{n-k}}{H^n}=\delta_k(f)$.

On the other hand, for $\delta_k (f \circ g)$ we first consider the birational transform $\mathcal{C}_{s_0}$ under $f^{-1}$ of a general element $H^k$, and $\mathcal{C}_{s_0}$ is now a special element in $B^k (X)$. We then have to compute the degree of the birational transform under $g^{-1}$ of $\mathcal{C}_{s_0}$, and divide it by $H^n$, and this is $\delta_k (f \circ g)$. 

By hypothesis, we can deform $\mathcal{C}_{s_0}$ in the homological equivalence class of $\delta_k (f) H^k$ to a general element $\mathcal{C}_{s}$ in that class via the family $p\colon \mathcal{C} \to S$ whose existence is assumed in Theorem \ref{tSubMultiplicative}. Then the birational transform, for general $s\in S$, under $g^{-1}$ of $\mathcal{C}_s$ is in $\delta_k(g) \delta_k(f)H^k$ since $\mathcal{C}_s$ is $g$-adapted and Lemma \ref{lCycleTransformBirTransform} (2) holds. But the birational transform under $g^{-1}$ of $\mathcal{C}_{s_0}$ may lie in $d H^k$ with $d \leq \delta_k (g) \delta_k (f)$ by Lemma \ref{lCyclesDeg}. Thus Theorem \ref{tSubMultiplicative} follows.
\end{proof}

\begin{remark}\xlabel{rCubicFourfolds}
For example, Theorem \ref{tSubMultiplicative} is applicable for a birational map $g \colon X \dasharrow X$ of a very general cubic fourfold, and a reflection $f = \sigma_p\colon X \dasharrow X$ in a point $p\in X$; in this case, $\mathcal{C}_{s_0}$ is the birational transform under $f$ of an element which is general in $H^2$, for some very ample divisor $H$. Then we can find an irreducible family $\mathcal{C}$ with the required properties, taking as elements of $\mathcal{C}$ the birational transforms of $2$-cycles $A$  under $\sigma_q$, for $q$ varying in $X$ and $A$ varying in the  system of complete intersections of two hyperplanes $H^2$. 
\end{remark}

\section{Degrees of iterated birational transforms of surfaces in fourfolds}\xlabel{sItTrans}

\subsection{An iterative set-up}
When computing dynamical degrees, especially $\lambda_2$ on fourfolds, one frequently has to understand how the degrees of successive birational transforms of some general surface in the initial variety change. 

We therefore study the following set-up: let $X$ be a smooth projective fourfold with a very ample divisor $H$. Let $f\colon X \dasharrow X$ be a birational map. We want to compute the first and second Cremona degrees of the iterates of $f$ successively. Let $i\in \ZZ$ be the iteration index. 

To begin with, we let $S_0\subset X$ be a surface which is the intersection of two very general elements in $H$. In particular, we assume $S_0$ is smooth. Let
\[
s_0 \colon S_0 \, \hookrightarrow X
\]
be the inclusion. Moreover, we write $h_0\subset S_0$ for the intersection with $S_0$ of a very general element of $H$. Then $h_0$ defines a very ample class on $S_0$. Moreover, we also put $D_0:= h_0$.

Now suppose inductively that $S_i$ and a morphism $s_i\colon S_i \to X$, together with divisors $h_i, D_i \subset S_i$, have already been defined. We then define $S_{i+1}$ and $s_{i+1}$, with $h_{i+1}$ and $D_{i+1}$, in the following way. Consider the diagram
\[
\xymatrix{
                &  S_{i+1} \ar[rd]^{\widetilde{s}_i}\ar[ldd]_{\sigma_i}\ar@/^2pc/[rrdd]^{s_{i+1}} &                                                                                &    \\
                &                 & \Gamma_f\ar[rd]^{\pi_2}\ar[ld]^{\pi_1}             &    \\
S_i\ar[r]_{s_i}\ar@{-->}[rru]^{\bar{s}_i}     &  X\ar@{-->}[rr]^f    &                                                                    & X
}
\]

Here $\sigma_i$ is a sequence of blow-ups of $S$ in reduced points such that the rational map $f\circ s_i$ becomes a morphism $s_{i+1}$ on the blown-up surface $S_{i+1}$. We will describe more precisely how to construct $\sigma_i$ in a moment. Once this is done, we define 
\[
h_{i+1}= \sigma_i^* (h_i), \quad D_{i+1} = s_{i+1}^* (H)
\]
and we would like to compute the quantities
\[
d_2^{(i+1)}:= D_{i+1}^2 , \quad d_1^{(i+1)} := D_{i+1}.h_{i+1}
\]
in terms of data associated to the resolution map $\sigma_i$ and of $d_2^{(i)}, \: d_1^{(i)}$. Note that $d_2^{(i+1)}$ is the second Cremona degree (with respect to the chosen $H$) of $f^{i+1}$, and $d_1^{(i+1)}$ is its first Cremona degree. 

Let us now say how $\sigma_i$ is defined. The map $f$ can be given by a certain linear system; i.e., there is a line bundle $\mathcal{L}$ on $X$, a subspace of sections $V \subset H^0 (X, \mathcal{L})$ and the associated linear system $|V| = \PP (V) \subset \PP (H^0 (X, \mathcal{L}))$ of divisors defining $f$. Recall that there is an evaluation morphism for sections $\mathrm{ev}_V \colon V \otimes \mathcal{O}_X \to \mathcal{L}$ which determines a morphism $V \otimes \mathcal{L}^{\vee}\to \mathcal{O}_X$ whose image is called the base ideal
\[
\mathfrak{b}(|V|) \subset \mathcal{O}_X
\]
of the linear system $|V|$. The closed subscheme $\mathrm{Bs}(|V|)\subset X$ it defines is called the base scheme of the linear system $|V|$. 

Since $S_0$ is chosen very general, the image of $s_i \colon S_i \to X$ is not contained in $\mathrm{Bs}(|V|)$ for any $i$, and the composite rational map $f\circ s_i$ is defined by the linear system attached to the pull-back space of sections $s_i^*(V) \subset H^0 (S_i, s_i^*\mathcal{L})$. Its base scheme is defined by the inverse image ideal sheaf $s_i^{-1}(\mathfrak{b}(|V|))$ which we call $\mathcal{I}_i$ for the sake of explaining how to construct $\sigma_i$ with simpler notation. Let $Z_i\subset S_i$ be the closed subscheme which $\mathcal{I}_i$ defines. 

Now we construct $\sigma_i$ inductively as a composite of point blow-ups as follows. Put $\mathcal{I}^{(0)}_i := \mathcal{I}_i$, $Z_i^{(0)}:= Z_i$ and $S_i^{(0)}:= S_i$. We wish to construct a sequence of morphisms
\[
\xymatrix{
S_{i+1}:= S_i^{(N)} \ar[r]^{\quad \;\; \pi_{N-1}} & \ldots \ar[r]^{\pi_1} & S_i^{(1)} \ar[r]^{\pi_0} & S_i^{(0)}
}
\]
where each $\pi_{j} \colon S^{(j+1)}_{i} \to S_i^{(j)}$ is a blow-up in a finite set of reduced points in $S_i^{(j)}$. 

\textbf{Step 1.} 
Suppose we have already constructed $S_i^{(j)}$ together with $\mathcal{I}^{(j)}_i$, $Z_i^{(j)}$. Then if $Z^{(j)}_i$ is zero-dimensional, we put $Z^{(j)}_i (\mathrm{point}) := Z^{(j)}_i$. If $Z^{(j)}_i$ is one-dimensional, then we can write 
\[
Z^{(j)}_i = \mathfrak{D}^{(j)}_i \cup \mathfrak{E}^{(j)}_i
\]
where $\mathfrak{D}^{(j)}_i$ is the union of the one-dimensional isolated components in a primary decomposition of $\mathcal{I}^{(j)}_i$. These define a unique divisor $ \mathfrak{D}^{(j)}_i$ whereas $\mathfrak{E}^{(j)}_i$, which may contain embedded point components, is not unique: recall for example that locally around the origin in $\mathbb{A}^2$, the ideal $(xy,y^2)$ can be written in many ways as an intersection of primary ideals:
\[
(xy,y^2) = (y) \cap (x^2, xy, y^2, x+ \alpha y) , \quad \alpha \in \CC 
\]
and the embedded point component defined by any $(x^2, xy, y^2, x+ \alpha y)$ is not unique at all. Let $\mathcal{O} (- \mathfrak{D}^{(j)}_i)$ be the ideal sheaf of $\mathfrak{D}^{(j)}_i$, and put
\[
Z^{(j)}_i (\mathrm{point}) := V \left( \left(\mathcal{I}^{(j)}_i : \mathcal{O} (- \mathfrak{D}^{(j)}_i) \right) \right)
\]
which is a well-defined ideal sheaf with support in a finite set of points. Note that whereas the embedded point components are not well-defined, the quotient ideal sheaf $ \left(\mathcal{I}^{(j)}_i : \mathcal{O} (- \mathfrak{D}^{(j)}_i) \right)$ is and defines a certain zero-dimensional subscheme $Z^{(j)}_i (\mathrm{point})$. In the example above, it is just the reduced origin corresponding to $(x,y)$. 
\smallskip

\textbf{Step 2.}
Define $\pi_j \colon  S^{(j+1)}_{i} \to S_i^{(j)}$ as the blow-up of $S_i^{(j)}$ in the reduced subscheme (a finite set of points) whose support agrees with $Z^{(j)}_i (\mathrm{point})$. Let 
\[
Z^{(j+1)}_i:=\pi_j^{-1} (Z^{(j)}_i (\mathrm{point}) )
\]
be the scheme-theoretic preimage with ideal sheaf $\mathcal{I}_i^{(j+1)}$. In other words, if $\mathcal{I}^{(j)}_i (\mathrm{point})$ is the ideal sheaf of $Z^{(j)}_i (\mathrm{point})$, the subscheme $Z^{(j+1)}_i\subset S_i^{(j+1)}$ is defined by the ideal sheaf which is the image of $\pi_j^*(\mathcal{I}^{(j)}_i (\mathrm{point}))$ in $\mathcal{O}_{S^{(j+1)}_i}$. This is, by definition, the inverse image ideal sheaf $\pi_j^{-1}(\mathcal{I}^{(j)}_i (\mathrm{point}))$. It is not necessarily principal of course, since we have blown up in the reduced subscheme underlying $Z^{(j)}_i (\mathrm{point})$, and not in $Z^{(j)}_i (\mathrm{point})$ itself. 

At any rate, it is easy to compute $Z^{(j+1)}_i$ in local coordinates on the blown-up surface simply by substituting these in the original ideal downstairs. With $Z^{(j+1)}_i$ now go back to Step 1. The process terminates since the length of the scheme $Z^{(j)}_i (\mathrm{point})$ drops strictly in each step until it becomes the empty set. 

It is now easy to describe what $D_{i+1}$ is on $S_{i+1}= S^{(N)}_{i+1}$. In particular, this then allows us to compute $d_2^{(i+1)}= D_{i+1}^2 , \: d_1^{(i+1)} = D_{i+1}.h_{i+1}$. If $D$ is a divisor which lives on some surface in the tower
\[
\xymatrix{
S_{i+1}:= S_i^{(N)} \ar[r]^{\quad \;\; \pi_{N-1}} & \ldots \ar[r]^{\pi_1} & S_i^{(1)} \ar[r]^{\pi_0} & S_i^{(0)}
}
\]
then we denote the pull-back of $D$ to the top floor $S_{i+1}$ of the tower simply by putting a hat on it: $\hat{D}$. Then
\[
D_{i+1} = \hat{D}_i - \sum_j \hat{\mathfrak{D}}^{(j)}_i .
\]

\subsection{Geometry of a reflection on a cubic fourfold}
As an illustration let us prove a result about equality of dynamical degrees when $X$ is a very general cubic fourfold and $f$ is a certain suitably general composition of reflections in points on $X$. We will say what suitably general means below.

First we recall some facts about the geometry of a reflection $\sigma_p \colon X \dasharrow X$ on a smooth cubic fourfold in a very general point $p$. Compare \cite[Sect.\ 3]{BBS15} for these. A general line in $\PP^5$ through $p$ intersects $X$ in two points away from $p$, and $\sigma_p$ is the birational involution interchanging these points. One main fact we will use is that the birational self-map $\sigma_p$ lifts to an automorphism $\widetilde{\sigma}_p$ on a suitable model $\widetilde{X}$:
\[
\xymatrix{
\widetilde{X}\ar[d] \ar[r]^{\widetilde{\sigma }_p} & \widetilde{X} \ar[d]  \\
X \ar@{-->}[r]_{\sigma_p}  &   X.
}
\]
and we can construct $\widetilde{X}$ in two steps $\widetilde{X} \to X' \to X$, where each map is a blow-up in a certain smooth center. See Figure \ref{figure1} for a schematic picture, which we will explain now in some more detail. 

\begin{figure}
\caption{Geometry of the resolution of a reflection.}\label{figure1}
\begin{center}
\includegraphics[height=150mm]{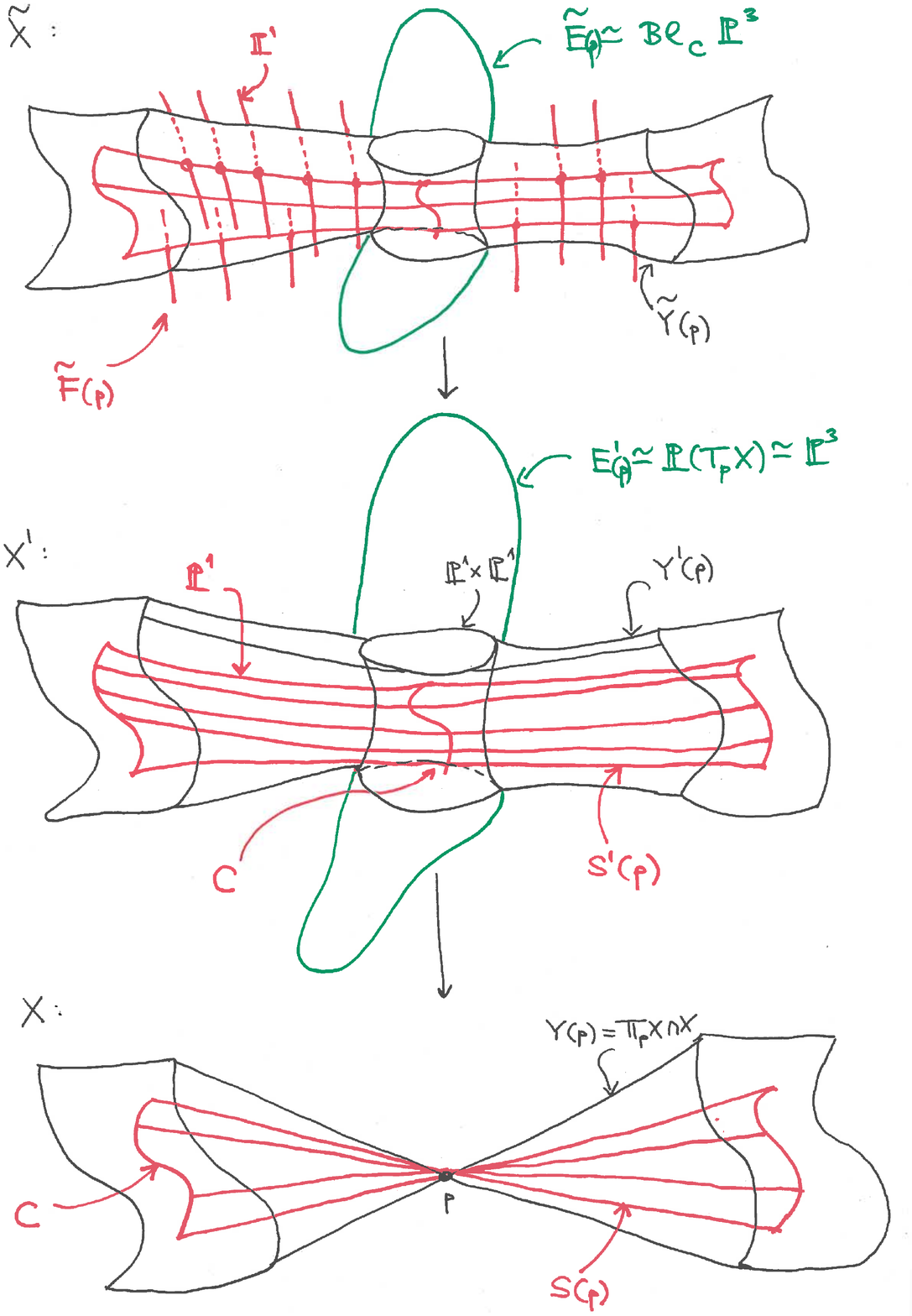}
\end{center}
\end{figure}

The embedded tangent hyperplane $\mathbb{T}_pX$ intersects $X$ in a cubic threefold $Y(p)$ with a node at $p$; the lines on $X$ through $p$ sweep out a surface $S(p)$ which coincides with the surface of lines on the nodal cubic threefold through $p$.  The tangent cone to $Y(p)$ in $p$ is a cone over a smooth quadric $Q \simeq \PP^1\times \PP^1$ in the hyperplane $\PP^3_{\infty} \subset \mathbb{T}_pX$, and $S(p)$  is a cone over a curve $C$ of bidegree $(3,3)$ in $Q$. The indeterminacy locus of $\sigma_p$ is $S(p)$, and $\sigma_p$ contracts $Y(p)$ to the point $p$. 

\begin{definition}\xlabel{dGoodPoint}
We call a point $p\in X$ \emph{good} if $C$ is a smooth curve.
\end{definition}

A general point $p$ will be good. To simplify, we will only consider reflections in good points in the sequel. Now $X' \to X$ is the blow-up of $X$ in $p$ with exceptional divisor $E'(p) \simeq \PP(T_pX) \simeq \PP^3$. Inside $E'(p)$ we retrieve $Q \simeq \PP^1\times \PP^1$ as the set of all tangent directions of smooth curve germs in $Y(p)$ through $p$, and  inside this, there is a copy of $C$, the set of directions of lines on $X$ through $p$.  The strict transform $Y'(p)$ of $Y(p)$ on $X'$ is the blow up of $Y(p)$ in the node, which gets replaced by $\PP^1\times \PP^1\simeq Q$. The strict transform $S'(p)$ of $S(p)$ on $X'$ is smooth because $p$ is a good point. 

Now $\widetilde{X} \to X'$ is the blow-up of $X'$ in $S'(p)$. We call the resulting exceptional divisor $\widetilde{F}(p)$. It is a $\PP^1$-bundle over $S'(p)$. The strict transform of $Y'(p)$ on $\widetilde{X}$ is called $\widetilde{Y}(p)$. The strict transform of $E'(p)$ on $\widetilde{X}$, denoted by $\widetilde{E}(p)$, is the blow-up of $E'(p)$ in $C$. We also denote by $\widetilde{H}$ the pull-back of a hyperplane section $H$ to $\widetilde{X}$ via the modification $\widetilde{X} \to X$. 

\begin{proposition}\xlabel{pFormulasReflection}
\begin{enumerate}
\item
The classes $\widetilde{H}, \widetilde{E}(p), \widetilde{F}(p)$ are a basis of $\mathrm{Pic}(\widetilde{X})$. Moreover, 
\begin{eqnarray}
\widetilde{Y}(p) \equiv \widetilde{H} - 2\widetilde{E}(p) - \widetilde{F}(p)
\end{eqnarray}
and 
\begin{eqnarray}
(\widetilde{\sigma}_p)^* (\widetilde{H}) \equiv & 2 \widetilde{H} - 3 \widetilde{E}(p) - \widetilde{F}(p)\\
(\widetilde{\sigma}_p)^* (\widetilde{E}(p)) \equiv & \widetilde{Y}(p)\equiv \widetilde{H} - 2\widetilde{E}(p) - \widetilde{F}(p) \\
(\widetilde{\sigma}_p)^* (\widetilde{F}(p)) \equiv & \widetilde{F}(p).
\end{eqnarray}
\item
We have
\begin{eqnarray}
(\widetilde{\sigma}_p)^* (\widetilde{H}) . \widetilde{Y}(p) \equiv & 0\\
\widetilde{H} . \widetilde{E}(p) \equiv & 0 .
\end{eqnarray}
\item
We have
\begin{eqnarray}
(\widetilde{\sigma}_p)^* (\widetilde{H})^2 \equiv & 2\widetilde{H}^2+3\widetilde{E}(p).\widetilde{E}(p) - \widetilde{F}(p).\widetilde{H} + \widetilde{E}(p).\widetilde{F}(p)\\
-\widetilde{Y}(p)^2 \equiv & \widetilde{H}^2 - 2(- \widetilde{E}(p)^2) - \widetilde{H}.\widetilde{F}(p) + \widetilde{E}(p).\widetilde{F}(p)
\end{eqnarray}
\item
If $h$ is any  divisor class on $\widetilde{X}$, then
\begin{eqnarray}
(\widetilde{\sigma}_p)^* (\widetilde{H}).h \equiv & 2 \widetilde{H}.h - 3 \widetilde{E}(p).h - \widetilde{F}(p).h\\
\widetilde{Y}(p).h \equiv & \widetilde{H}.h - 2\widetilde{E}(p).h - \widetilde{F}(p).h
\end{eqnarray}
\end{enumerate}
\end{proposition}

\begin{proof}
Part (1) is straightforward and can be found in \cite[Sect.\ 3, Prop.\ 3.2 ff]{BBS15}. As for (2), note that by the projection formula
\[
(\widetilde{\sigma}_p)^* (\widetilde{H}) . \widetilde{Y}(p) \equiv \widetilde{H} . (\widetilde{\sigma}_p)_* (\widetilde{Y}(p)) \equiv \widetilde{H}.\widetilde{E}(p) \equiv 0
\]
since a general hyperplane section $H$ of $X$ does not meet $p$. 

For (3) we compute 
\begin{align*}
(\widetilde{\sigma}_p)^* (\widetilde{H})^2  & \equiv   (\widetilde{\sigma}_p)^* (\widetilde{H}). (\widetilde{\sigma}_p)^* (\widetilde{H}) - (\widetilde{\sigma}_p)^* (\widetilde{H}). \widetilde{Y}(p)\\
                                                          &  \equiv  (\widetilde{\sigma}_p)^* (\widetilde{H}) . ( (\widetilde{\sigma}_p)^* (\widetilde{H}) - \widetilde{Y}(p))\\
                                                          & \equiv  (2 \widetilde{H} - 3 \widetilde{E}(p) - \widetilde{F}(p) ) . (\widetilde{H} - \widetilde{E}(p))\\
                                                          & \equiv  2 \widetilde{H}.\widetilde{H} - 3 \widetilde{H}.\widetilde{E}(p) - \widetilde{H}.\widetilde{F}(p) - 2\widetilde{H}.\widetilde{E}(p) + 3 \widetilde{E}(p).\widetilde{E}(p) + \widetilde{F}(p).\widetilde{E}(p)\\
                                                          & \equiv  2\widetilde{H}.\widetilde{H}+3\widetilde{E}(p).\widetilde{E}(p) - \widetilde{F}(p).\widetilde{H} + \widetilde{E}(p).\widetilde{F}(p) 
\end{align*}
where one uses the second formula of (2), $\widetilde{H} . \widetilde{E}(p) \equiv 0$ in the last step. 
Now
\begin{align*}
-\widetilde{Y}(p)^2 & \equiv  (\widetilde{\sigma}_p)^* (\widetilde{H}) . \widetilde{Y}(p) - \widetilde{Y}(p)^2\\
                          & \equiv ((\widetilde{\sigma}_p)^* (\widetilde{H}) - \widetilde{Y}(p) ) . \widetilde{Y}(p)\\
                           & \equiv  (\widetilde{H} - \widetilde{E}(p)). (\widetilde{H} - 2\widetilde{E}(p) - \widetilde{F}(p))\\
                           & \equiv  \widetilde{H}^2 - 2\widetilde{H}.\widetilde{E}(p) - \widetilde{H}.\widetilde{F}(p) - \widetilde{E}(p).\widetilde{H}(p) + 2\widetilde{E}(p)^2 + \widetilde{E}(p).\widetilde{F}(p)\\
                          & \equiv  \widetilde{H}^2 + 2\widetilde{E}(p)^2 - \widetilde{H}.\widetilde{F}(p) + \widetilde{E}(p).\widetilde{F}(p).
\end{align*}                     
Finally, part (4) is clear from (5.1) and (5.2).
\end{proof} 

Now we pass to an iterative set-up again. Namely, consider a sequence of points $\{ p_i \}_{i =1, 2, \dots }$ on $X$ and the corresponding sequence of reflections $\{ \sigma_{p_i} \}_{i=1, \dots }$ 
Take a surface $S_0 \subset X$ which is the intersection of two very general elements in $H$, and let $h_0\subset S_0$ be the intersection of $S_0$ with a very general element in $H$. Thus $S_0$ and $h_0$ are, respectively, a smooth cubic surface and a smooth cubic curve on $X$. We want to study the successive birational transforms of $S_0$ and $h_0$ when we apply $\sigma_{p_1}$, then $\sigma_{p_2}$, $\sigma_{p_3}$ and so forth. In particular, we want to study the degrees of those birational transforms. Let
\[
d_1^{(i)}, \quad d_2^{(i)}
\]
be the degrees, respectively, of the birational transforms of $h_0$ resp.\ $S_0$ under $\sigma_{p_i}\circ \dots \circ \sigma_{p_1}$. To study iterates of a single map, let us assume that the sequence of points is \emph{periodic with period} $N$, i.e.
\[
p_{i+N} = p_i\quad \forall i.
\]
Let us also assume that all the points $p_i$ are \emph{good}. 
The growth rates of the degrees $d_1^{(i)}$ resp. $d_2^{(i)}$ for $i\to \infty$ are then the first resp.\ second dynamical degrees of the composite
\[
\sigma_{p_N} \circ \dots \circ \sigma_{p_1}. 
\]
To compute the degrees, we first define a sequence of auxiliary surfaces $S_i$ with morphisms $s_i \colon S_i \to X$ inductively as follows: for $i=0$, $S_0$ has already been defined, and $s_0$ is the inclusion into $X$. Suppose now $s_i \colon S_i \to X$ has been defined. Look at the commutative diagram
\begin{equation}
\xymatrix{
S_{i+1}\ar[dd]_{\pi_{i+1}}\ar[rd]^{\widetilde{s}_i}\ar@/^6pc/[rrdd]^{s_{i+1}} &  &  \\
&  \widetilde{X}_{i+1}\ar[d]^{\mu_{i+1}} \ar[r]^{\widetilde{\sigma }_{p_{i+1}}} & \widetilde{X}_{i+1} \ar[d]_{\mu_{i+1}}  \\
S_i \ar[r]^{s_i}\ar@{-->}[ru] & X \ar@{-->}[r]_{\sigma_{p_{i+1}}}  &   X.
}
\end{equation}
Here $\mu_{i+1} \colon \widetilde{X}_{i+1} \to X$ is the modification described above: blow-up the point $p_{i+1}$ on $X$ and then the strict transform of the surface of lines through $p_{i+1}$. The morphism $\pi_{i+1} \colon S_{i+1} \to S_i$ is a composite of blow-ups in reduced points, constructed in the way described at the beginning of this section, such that the diagonal dotted arrow becomes a morphism $\widetilde{s}_i$. Then $s_{i+1}$ is simply the composite
\[
s_{i+1} =  \mu_{i+1}\circ \widetilde{\sigma}_{p_{i+1}} \circ \widetilde{s}_i .
\]
Now, in the notation of Proposition \ref{pFormulasReflection}, we have divisor classes $\widetilde{H}_{i+1}, \widetilde{E}(p)_{i+1},$ $ \widetilde{F}(p)_{i+1}, \widetilde{Y}(p)_{i+1}$ and $(\widetilde{\sigma}_{p_{i+1}})^*(\widetilde{H}_{i+1})$ on $\widetilde{X}_{i+1}$ and we give names to their pull-backs via $\widetilde{s}_i$ to $S_{i+1}$:
\begin{eqnarray*}
D_{i+1}:= & (\widetilde{s}_i)^* ((\widetilde{\sigma}_{p_{i+1}})^*(\widetilde{H}_{i+1}) ) \\
E_{i+1}:= & (\widetilde{s}_i)^* (\widetilde{E}(p)_{i+1}) \\
F_{i+1}:= & (\widetilde{s}_i)^* ( \widetilde{F}(p)_{i+1}) \\
Y_{i+1}:=& (\widetilde{s}_i)^* ( \widetilde{Y}(p)_{i+1}).
\end{eqnarray*}
Note that by the commutativity of the diagram (5.11), the equality
\begin{align*}
(\widetilde{s}_i)^* (\widetilde{H}_{i+1})&=(\widetilde{s}_i)^* (\mu_{i+1}H)=\pi_{i+1}^*s_i^*H\\
&=\pi_{i+1}^*(\widetilde{s}_{i-1}^*\widetilde{\sigma}_{p_i}^*\mu_i^*H)=\pi_{i+1}^* (D_i) 
\end{align*}
holds, so we do not need a new name for $(\widetilde{s}_i)^* (\widetilde{H}_{i+1})$. Then the morphism $s_i \colon S_i \to X$ is defined by the linear system $|D_i|$, i.e. $s_i^*(H) = D_i$. Moreover, we make the following simplifying notational convention: if some divisor class $D$ lives on $S_j$, then for any $i > j$ we denote the pull-back of $D$ to $S_i$ via the composite
\[
\xymatrix{
S_i \ar[r]^{\pi_i} & \dots \ar[r]^{\pi_{j+1}} & S_j
}
\]
simply by the same letter $D$. Thus, for example, $h_0\subset S_0$ defines a class on each surface $S_i$ by pull-back, and we denote all of them by the same letter, where, in a given equation, the context makes it clear on which $S_i$ this holds. 

Now all the equations (5.1) to (5.10) give equations on any surface $S_{i+1}$ by pulling the divisors back via $\widetilde{s}_i$. However, we want to impose a certain genericity condition on the points $p_1, \dots , p_N$ such that these equations take a simpler form.

\begin{assumption}\xlabel{aPoints}
The $N$-periodic sequence of good points $p_i$ can be chosen such that for a very general $S_0$ and $h_0$, the following equations hold on any blown-up surface $S_{i+1}$:
\begin{eqnarray*}
F_{i+1}. D_i = 0 , \quad E_{i+1}.F_{i+1} = 0.
\end{eqnarray*}
\end{assumption}

In fact we will assume something a little bit stronger, which implies the previous assumption, but can be expressed more geometrically in terms of the successive birational transforms of $S_0$:
\begin{assumption}\xlabel{aPointsGeometric}
The $N$-periodic sequence of good points $p_i$ can be chosen such that for a very general $S_0$ and $h_0$, the surface $s_i(S_i) \subset X$ is $\sigma_{p_{i+1}}$-adapted in the sense Lemma \ref{lCycleTransformBirTransform} (2) except for the following phenomenon which we allow to cause failure of $\sigma_{p_{i+1}}$-adaptedness: for a point $p_{k}$ there can be a point $p_{\tau (k)}$ such that all points
\[
\sigma_{p_l}\circ \dots \circ \sigma_{p_{k+1}}(p_k), \quad k\le l \le \tau (k)-2
\]
land in the open set where $\sigma_{p_{l+1}}$ is a local isomorphism onto the image, and 
\[
p_{\tau (k)}= \sigma_{p_{\tau (k) -1}}\circ \dots \circ \sigma_{p_{k+1}}(p_k).
\]
Moreover, the birational transform $s_{\tau(k)-1}(S_{\tau(k)-1}) \subset X$ is such that its strict transform in $X'_{\tau (k)}$ (obtained from $X$ by blowing up the point $p_{\tau (k)}$) meets the exceptional divisor in a curve, $\Gamma_{\tau (k)}$ say, which \emph{does not coincide with} the curve $C_{\tau (k)}$ of directions of lines through $p_{\tau (k)}$. In fact, this condition ensures that $E_{\tau (k)}. F_{\tau (k)} =0$ on $S_{\tau (k)}$: this is clear if $\Gamma_{\tau (k)}$ and $C_{\tau (k)}$ are even distinct; but if they meet in only finitely many points, then, looking back at Figure 1, we find that $S_{\tau (k)}$ can be viewed as a blow-up $\pi\colon S_{\tau (k)} \to S'_{\tau (k)}$ of a surface $S'_{\tau (k)}$ such that $E_{\tau (k)}$ is a pull-back of a divisor on $S'_{\tau (k)}$ and $F_{\tau (k)}$ is exceptional for $\pi$.
\end{assumption}

Intuitively, we can express Assumption \ref{aPointsGeometric} by saying that we allow a curve in a certain birational transform of $S_0$ to be contracted to a reflection point $p_k$ at some stage of the iteration, and this point then wanders around, always staying in the domain of definition and away from the exceptional locus of the respective next reflection, until at some later time it gets mapped to $p_{\tau (k)}$ after applying $\sigma_{p_{\tau(k)-1}}$. We then also assume that the directions of smooth curve germs which lie on the birational transform $s_{\tau(k)-1}(S_{\tau(k)-1}) \subset X$ and which pass through $p_{\tau(k)}$ are not identical with the directions of lines through $p_{\tau (k)}$. 

In this way, we will always have $E_{i+1}.F_{i+1}=0$ under Assumption \ref{aPointsGeometric} as well as $F_{i+1}. D_i = 0$ since the birational transforms of $S_0$ intersect in points the surfaces of lines attached as indeterminacy loci to subsequent reflections. Thus Assumption \ref{aPointsGeometric} is stronger than Assumption \ref{aPoints}. Moreover, we then have the \emph{successor function} $\tau \colon \mathbb{N} \to \mathbb{N}\cup \{ \infty \}$, which associates to $k$ the value $\tau (k)$ if $p_k$ gets mapped to $p_{\tau (k)}$ after a while as in Assumption \ref{aPointsGeometric}, or is equal to $\infty$ if there is no successor (if the point never gets mapped unto another reflection point). 

For the time being we will not discuss whether the genericity assumption \ref{aPointsGeometric} can be satisfied by a certain configuration of points $p_1, \dots , p_N$ on $X$. Also, this is not so important from our point of view. Rather, we would like to demonstrate that one can prove that some property (P) of dynamical degrees forces special geometric configurations on $X$, or conversely, that for suitably general configurations, the dynamical degrees fail to have property (P), if the configuration is realizable. The following result is a sample for this.

\begin{theorem}\xlabel{tGeneralPointsEquality}
Suppose that $\{ p_i \}$ is an $N$-periodic sequence of good points on $X$ which satisfy Assumption \ref{aPointsGeometric}. Then the first and second dynamical degrees of the map
\[
\sigma_{p_N} \circ \dots \circ \sigma_{p_1}
\]
are equal.
\end{theorem}

\begin{proof}
Under the hypotheses, we have the equations on $S_{i+1}$
\begin{eqnarray*}
D_{i+1}^2 = 2 D_{i}^2 + 3 E_{i+1}^2\\
-Y_{i+1}^2 =  D_i^2 - 2(-E_{i+1}^2)
\end{eqnarray*}
from (5.7) and (5.8), and from (5.9) and (5.10) 
\begin{eqnarray*}
D_{i+1}.h_0 = 2D_{i}.h_0 - 3 E_{i+1}.h_0\\
Y_{i+1}.h_0 = D_i.h_0 - 2E_{i+1}.h_0
\end{eqnarray*}
since $F_{i+1}.h_0 =0$ because $F_{i+1}$ lies over points in $S_0$. Note that these equations are valid universally, but $E_{i+1}$ may very well be zero, and indeed will be unless $i+1$ is of the form $\tau (k)$. Now note that
\[
d_2^{(i+1)}= D_{i+1}^2, \quad d_1^{(i+1)} = D_{i+1}.h_0
\]
and abbreviate 
\[
t_2^{(i+1)}:= -Y_{i+1}^2, \quad t_1^{(i+1)}:= Y_{i+1}.h_0.
\]
Moreover, note that under our geometric Assumption \ref{aPointsGeometric}, we have
\[
E_{\tau (k)} = Y_k. 
\]
Thus we get the recursions
\begin{eqnarray}
d_2^{(i+1)} = 2 d^{(i)}_2 - 3 t_2^{(\tau^{-1}(i+1))}\\
t_2^{(i+1)} =  d_2^{(i)} - 2 t_2^{(\tau^{-1}(i+1))}
\end{eqnarray}
where we agree that we put
\[
t_2^{(\tau^{-1}(i+1))} := 0 \; \mathrm{if} \; \tau^{-1}(i+1)= \emptyset 
\]
by definition. Also, 
\begin{eqnarray}
d_1^{(i+1)} = 2 d^{(i)}_1 - 3 t_1^{(\tau^{-1}(i+1))}\\
t_1^{(i+1)} =  d_1^{(i)} - 2 t_1^{(\tau^{-1}(i+1))}
\end{eqnarray}
Note that $\tau^{-1}(1) = \emptyset$.
Hence, the two sequences $d^{(i)}_1$ and $d^{(i)}_2$ are determined by the same set of recursions, starting from $d^{(0)}_1=d^{(0)}_2=3$, $t_1^{(0)}=t_2^{(0)}=0$, thus coincide. In particular, the first and second dynamical degrees of $\sigma_{p_N} \circ \dots \circ \sigma_{p_1}$ are equal. 
\end{proof}

\begin{remark}\xlabel{rRecursive}
In the set-up of Theorem \ref{tGeneralPointsEquality}, put 
\[
C:= \mathrm{l.c.m.}_{i=1, \dots , N} \{ N, \tau (1) -1, \tau (2) -2, \dots , \tau (N) -N \}.
\]
Then the recursions (5.11), (5.12) resp. (5.13), (5.14) show that there is a $2C\times 2C$ matrix $M$ (with constants as entries) such that for for the vectors
\begin{eqnarray*}
v_1^{(i)} = (d^{(C+iC)}_1 , \dots , d_1^{(1+iC)}, t_1^{(C+iC)}, \dots , t_1^{(1+iC)})\\
v_2^{(i)} = (d^{(C+iC)}_2 , \dots , d_2^{(1+iC)}, t_2^{(C+iC)}, \dots , t_2^{(1+iC)})
\end{eqnarray*}
we have
\[
v_1^{(i+1)} = Mv_1^{(i)}, \quad v_2^{(i+1)} = M v_2^{(i)}.
\]
This can be used to compute the dynamical degrees for concretely given successor functions $\tau$. Under certain generality assumptions, they will be equal to the spectral radius of $M$. 
\end{remark}

\section{Appendix: Inner product structures and generalized Picard-Manin spaces}\xlabel{sInnerProduct}

Another strong source of inequalities for dynamical degrees is the phenomenon of \emph{hyperbolicity}; Picard-Manin spaces and associated hyperbolic spaces have so far been studied mainly for divisors on surfaces by Cantat, Blanc et al., see also \cite[Sect.\ 2]{Xie15}, for a survey. We want to show that something similar can be done under much more general circumstances, using cycles of higher codimension and the Hodge-Riemann bilinear relations/the Hodge index theorem in higher dimensions. For definiteness, we will deal with fourfolds $X$ here only, and consider $H^2(X)$, i.e., codimension $2$ algebraic cycles modulo homological equivalence, on them. However, the inner product structures we will produce on our infinite-dimensional spaces will be more complicated than Euclidean or hyperbolic.

\medskip


For a smooth projective fourfold $X$ and a surjective birational morphism $\pi \colon Y \to X$ of another smooth projective fourfold $Y$ onto $X$, where we also assume that $\pi$ is a succession of blow-ups along smooth centers, we can consider the induced linear maps on cycle classes (which we take with real coefficients for convenience now)
\[
\pi_* \colon H_{\RR}^2(Y) \to H_{\RR}^2(X), \quad \pi^* \colon H_{\RR}^2(X) \to H_{\RR}^2(Y).
\]
We will write $Y \ge X$ and say that $Y$ dominates $X$.

\begin{definition}\xlabel{dCycleLimits}
The generalized Picard-Manin spaces are given as follows:
as a projective limit with respect to the push-forward maps
\[
H^2(X)_{\mathrm{proj}} := \varprojlim\limits_{Y \ge X} H_{\RR}^2(Y)
\]
or as an injective limit using the pull-back maps
\[
H^2(X)_{\mathrm{inj}} := \varprojlim\limits_{Y \ge X} H_{\RR}^2(Y).
\]
\end{definition}

The projection formula shows that there is an injection $H^2(X)_{\mathrm{proj}} \subset H^2(X)_{\mathrm{proj}}$; an analogous result holds for curves on surfaces, where the space constructed via the injective limit is sometimes called the space of Cartier classes on the Zariski-Riemann space (or Picard-Manin space), and the projective limit is called the space of Weil classes. We will work with the injective limit construction in the sequel.

Each of the spaces $H_{\RR}^2(Y)$ carries an inner product $(\cdot , \cdot )$, the (nondegenerate) \emph{intersection form}.  The pull-back maps $\pi^*$ are isometries for birational proper maps $Y \to X$. Hence $H^2(X)_{\mathrm{inj}}$ carries the structure of an infinite-dimensional inner product space $\mathbb{E}(X)$. The advantage of this is that any birational map $f \colon X \dasharrow X$ induces an \emph{isometry} $f^*_{\mathbb{E}(X)}$ of $\mathbb{E}(X)$: if an element $\alpha \in H^2(X)_{\mathrm{inj}}$ is represented by a class $\alpha_1 \in H^2(X_1)_{\RR}$ where $\pi\colon X_1\to X$ dominates $X$, then there is a model $p\colon \widetilde{X}_1 \to X$ such that $\widetilde{f} := \pi^{-1} \circ f \circ p \colon \widetilde{X}_1 \to X_1$ is a morphism. Then $\widetilde{f}^* (\alpha_1)$ represents the image under $f^*_{\mathbb{E}(X)}$ of $\alpha_1$ in $\mathbb{E}(X)$. This gives a well-defined map since for any two models we can find a third dominating both of them. Also 
\[
f^*_{\mathbb{E}(X)}\colon \mathbb{E}(X) \to \mathbb{E}(X)
\]
is clearly an isometry. 

\medskip

Note that for an ample class $h$ on $X$, the dynamical degrees of $f$ are given by the growth behavior of the inner products on $\mathbb{E}(X)$:
\[
( (f^*_{\mathbb{E}(X)})^n (h^k) , h^{4-k})
\]
so the dynamics of these isometries is important to study. 
\smallskip

Let us investigate the signature of the inner product/non-degenerate bilinear form on $\mathbb{E}(X)$ more concretely: first, for the $H^{2,2}$-part of the middle cohomology of a fourfold, we have the (orthogonal with respect to the intersection form) Lefschetz decomposition
\[
H^{2,2} = L^0 H^{2,2}_{\mathrm{prim}} \oplus L^1 H^{1,1}_{\mathrm{prim}} \oplus L^2 H^{0,0}_{\mathrm{prim}}.
\]
Here $L$ is the Lefschetz operator, and the subscript prim denotes primitive cohomology. 
Moreover, the intersection form on $L^r H^{a,b}_{\mathrm{prim}}$ is definite of signature $(-1)^a$, and $h^{a,b}_{\mathrm{prim}} = h^{a,b} - h^{a-1, b-1}$. Hence the signature on $H^{2,2}(X)$ of a fourfold $X$ is
\begin{equation}
h^{2,2}_{\mathrm{prim}} - h^{1,1}_{\mathrm{prim}} + h^{0,0}_{\mathrm{prim}} = h^{2,2} - 2 h^{1,1} + 2.
\end{equation}

In particular, if we start with a cubic fourfold, we obtain a positive intersection product on $H^{2,2}$. Now we start blowing up along points, curves and surfaces to obtain another model $Y \to X$, dominating $X$. Let us see how this affects the initial signature.

For a blow-up $\widetilde{X}_Z$ of $X$ in a smooth center $Z$ of codimension $r$ we have the decomposition of Hodge structures, see, for example, \cite[Thm.\ 7.31]{Voi03}:
\[
H^4 (\widetilde{X}_Z, \ZZ ) = H^4(X, \ZZ ) \oplus \bigoplus_{i=0}^{r-2} H^{4-2i-2}(Z, \ZZ ),
\]
where we shift the weights in the Hodge structure on $H^{4-2i-2}(Z, \ZZ )$ by $(i+1, i+1)$ to obtain a Hodge structure of weight $4$ (and endow the intersection form on it with a sign if we want the decomposition to be compatible with inner products). We also always have the equality $h^{1,1}(\widetilde{X}_Z) = h^{1,1}(X) +1$. Now suppose

\begin{enumerate}
\item
$Z$ is a point: then $h^{2,2}(\widetilde{X}_Z) = h^{2,2}(X) +1$, and by (4.1), the inner product, if we imagine it to be diagonalized over $\RR$ to be given by a matrix with $+1$'s and $-1$'s on the diagonal, changes by adding one copy of $-1$. 
\item
$Z$ is a curve $C$: then $h^{2,2}(\widetilde{X}_Z) = h^{2,2}(X) +h^{1,1}(C) + h^{0,0}(C) = h^{2,2}(X) +2$, and by (4.1) the inner product on the new vector space (which is two dimensions bigger) changes by adding one copy of $+1$ and one of $-1$.
\item
$Z$ is a surface: then $h^{2,2}(\widetilde{X}_Z) = h^{2,2}(X) +h^{1,1}(S)$, the new $H^{2,2}$ is bigger by $h^{1,1}(S)$ dimensions, and the inner product changes by adding $h^{1,1}(S)-1$ entries $+1$ and one entry $-1$. 
\end{enumerate}

We can also describe the process on our algebraic cycles $H^2(X)$ (Hodge classes, since the Hodge conjecture holds for a cubic fourfold) now: if $Y \to X$ is obtained by blowing up a point, then $H^2(Y)$ is one dimension bigger than $H^2(X)$ and the intersection form described by adding a $-1$ along the diagonal; if it is obtained by blowing up a curve, then $H^2(Y)$ is two dimensions bigger, and we add a $+1$ and a $-1$ along the diagonal for the new intersection product; if $Y$ is obtained by blowing up a surface, then $H^2(Y)$ is bigger by the Picard rank $\rho$ of the surface, and we add $\rho -1$ entries $+1$ and one entry $-1$ along the diagonal for the new intersection product.

Let $E\simeq \PP(\mathcal{N}_{Z/X})$ be the exceptional divisor and $p\colon E \to Z$ the induced map. The extra cycles in $H^2 (Y)$ can be described as the pull-backs via $p$ of (1) the point $Z$ if $Z$ is a point, (2) the curve $Z$ and a point on it if $Z$ is a curve, (3) curves on $Z$ if $Z$ is a surface, each time intersected with an appropriate power $H_E^i$ of the relative hyperplane class $H_E$ of the projective bundle $E$ to get an algebraic two-cycle. Thus, if $Z$ is a surface for example, we pull back curves on it to $E$, which is a $\PP^1$-bundle, to get surfaces on $Y$.  If $Z$ is a curve, $E$ is a $\PP^2$-bundle over it, and we get two extra surfaces as the class of a fiber and $H_E$ (pushed forward to $X$). For a point, we get a $\PP^3$-bundle $E$ over it, and one additional algebraic $2$-cycle, namely $H_E$, pushed forward to $X$. 

\smallskip

The situation is thus more complicated than the one with Picard-Manin spaces for surfaces, since there one can only blow up points, always adding $-1$'s to the intersection form, which makes the resulting limit into a hyperbolic space in the sense of Gromov. As we saw above, here we can be forced to add $+1$'s and $-1$'s, depending on whether we blow up points, curves or surfaces. Hopefully the fact that on $\mathbb{E}(X)$ we achieve some sort of algebraic stability for the map $f$ (taking iterates commutes with passing to associated maps on $\mathbb{E}(X)$), and the fact that we can describe the inner product geometrically in the above sense, can help to prove estimates for dynamical degrees in terms of the Cremona degrees in several cases.



\begin{thebibliography}{9999999999}

\bibitem[B-C13]{B-C13} J.\ Blanc and S.\ Cantat, 
\emph{Dynamical degrees of birational transformations of projective surfaces}, 
preprint (2013), arXiv:1307.0361 [math.AG]

\bibitem[BBB15]{BBB15}
F.\ Bogomolov, C.\ B\"ohning and H.-C.\ Graf von Bothmer, 
\emph{Birationally isotrivial fiber spaces}, to appear in Eur.\ J.\ Math. (2015)

\bibitem[BBS15]{BBS15}
C.\ B\"ohning, H.-C.\ Graf von Bothmer and P.\ Sosna, 
\emph{On the dynamical degrees of reflections on cubic fourfolds}, 
preprint (2015), arXiv:1502.01144 [math.AG]

\bibitem[EiHa15]{EiHa15}
D.\ Eisenbud and J.\ Harris, 
\emph{3264 \& All That Intersection Theory in Algebraic Geometry}, forthcoming book, available at \url{https://canvas.harvard.edu/files/419720/download?download_frd=1&verifier=ViDocWAkVVQhh7azATFF35h9K1bqLz6Udtlgqqwq}


\bibitem[Ful98]{Ful98}
W.\ Fulton, 
\emph{Intersection Theory}, 
Second Edition, Springer-Verlag (1998)

\bibitem[Guedj10]{Guedj10}
V.\ Guedj,
\emph{Propri\'{e}t\'{e}s ergodiques des applications rationnelles}, 
Quelques aspects des syst\`{e}mes dynamiques polynomiaux,  
Panor.\ Synth\`{e}ses \textbf{30}, Soc.\ Math.\ France, Paris (2010), 97--202

\bibitem[Ros56]{Ros56} M.\ Rosenlicht,
\emph{Some basic theorems on algebraic groups},
Amer.\ J.\ Math.\ \textbf{78} (1956), 401--443

\bibitem[RS97]{RS97}
A.\ Russakovskii and B.\ Shiffman, 
\emph{Value distribution for sequences of rational mappings and complex
dynamics}, 
Indiana Univ.\ Math.\ J.\ \textbf{46} (1997), 897--932

\bibitem[Truo15]{Truo15}
T.\ T.\ Truong, 
\emph{(Relative) dynamical degrees of rational maps over an algebraic closed field}, 
preprint (2015), arXiv:1501.01523 [math.AG]

\bibitem[Voi03]{Voi03}
C.\ Voisin, \emph{Hodge Theory and Complex Algebraic Geometry I}, 
Cambridge studies in advanced mathematics \textbf{77}, Cambridge University Press (2003)

\bibitem[Xie15]{Xie15}
J.\ Xie, 
\emph{Periodic points of birational transformations on projective surfaces}, 
Duke Math.\ J.\ \textbf{164}, Nr.\ 5 (2015), 903--932

\end{thebibliography}
\end{document}